%% file: Profinite_Sheaves.tex
\documentclass[a4paper]{article}
\input{macros}
\input{environs}

\title{Pontryagin duality and sheaves of profinite modules}
\author{Gareth Wilkes}
\numberwithin{equation}{section}

\begin{document}
\maketitle
\begin{abstract}
The well-known theory of Pontryagin duality provides a strong connection between the homology and cohomology theories of a profinite group in appropriate categories. A construction for taking the `profinite direct sum' of an infinite family of profinite modules indexed over a profinite space has been found to be useful in the study of homology of profinite groups, but hitherto the appropriate dual construction for studying cohomology with coefficients in discrete modules has not been studied. This paper remedies this gap in the theory.
\end{abstract}

\section*{Introduction}

While the study of discrete groups necessitates the use of both homology and cohomology theories for different purposes, for profinite groups one can generally select one favoured theory in which to work: the Pontryagin dual $M^\vee = \Hom(M, \R/\Z)$ obeys the strong relation
\[H^n(G,M^\vee) = H_n(G,M)^\vee\]
for any profinite group $G$ and profinite $G$-module $M$. See \cite[Section 5.1]{RZ}, \cite[Section 6.2]{wilkesbook}. The functor $(-)^\vee$ thus provides a duality between the categories of profinite $G$-modules and discrete torsion $G$-modules which greatly simplifies the study of the (co)homology theory of the profinite group $G$.

For profinite modules, the operation of `profinite direct sum' of a family (`sheaf', see Definition~\ref{defn:sheafpromodules} below) of profinite modules continuously indexed over a profinite space is well-established and is useful in studying the subgroup structure of a profinite group. One prominent example is the Mayer-Vietoris sequence for the homology of a profinite group acting on a profinite tree \cite[Chapter 9]{ribes}. 

In view of the duality between homology and cohomology for a profinite group, one would expect there to exist a dual construction which may be used to study the cohomology of a family of discrete torsion modules indexed continuously over a profinite space. Hitherto this seems to have been oddly absent from the literature. Indeed, some papers (e.g.\ \cite{haran2023relatively}, \cite[Section 10.5]{ribes}, \cite{wilkes2022relative}) which might naturally call upon such a theory have made some awkward dualisations to avoid studying the product of an infinite family of modules.

In this paper we rectify this gap in the theory by explaining fully the duality between the categories of sheaves of profinite and discrete modules over a profinite space $X$, and describing the dual construction to a profinite direct sum. The crucial connection is the use of \'etale spaces: the appropriate dual object to a `sheaf' of profinite modules $\cM$ is an \'etale space $\cA$ of discrete modules, and the space of global sections $\cprod \cA$ of such an \'etale space is the dual construction to the profinite direct sum $\bigboxplus \cM$. The explicit formal statements may be found as Theorems~\ref{thm:maintheorem} and \ref{thm:maintheorem2} below; I will not repeat it here to avoid cluttering the introduction with myriad definitions of categories. 

\begin{cnv}
In this paper, we will use the following notational conventions. 
\begin{itemize}
\item Categories and functors will be written in sans-serif font, e.g.\ \etale. For a category $\cat{C}$, the opposite category is denoted \op{\cat{C}}. 
\item Natural transformations between functors will receive capital Greek letters.
\item Cursive script will denote objects viewed as indexed families of objects, typically modules. Individual modules will receive capital Roman letters.
\item A disjoint union $\bigsqcup_{x\in X} A_x$ is defined as the union $\bigcup_{x\in X} A_x \times \{x\}$. Thus $(a,x)\in \bigsqcup_{x\in X}A_x$ defines an element $a\in A_x$.
\item For spaces $Y$ and $Z$ equipped with maps $f\colon Y \to X$ and $g\colon Z \to X$, the symbol $Y \times_X Z$ denotes the pushout
\[Y\times_X Z = \{(y,z)\in Y\times Z \mid  f(y) = g(z)\}.\]
\item Throughout, $R$ denotes a (perhaps non-commutative) profinite ring with unity. The category of discrete\footnote{Note that since $R$ has a $1$, a discrete $R$-module will be torsion as an abelian group.} right $R$-modules will be called \DMod, and the category of profinite left $R$-modules will be called \PMod. 
\end{itemize}
\end{cnv}

\section{Sheaves and cosheaves over profinite spaces}

Sheaf theory, broadly construed, seeks to study a topological space by considering functors from its category of open sets into some well-understood coefficient category, typically the category of abelian groups \cat{Ab}. For our study of the profinite ring $R$ we will modify this in two ways: we change the coefficient category into a suitable category of $R$-modules, and we allow the space to vary. 

Let $X$ be a profinite space. Rather than the category of all open sets in $X$, it is more appropriate for our needs to work with clopen sets. Let \OX\ be the poset category whose objects are the compact open subsets of $X$. 

\begin{defn}
A {\em presheaf of discrete $R$-modules} is a pair $(\cat{A}, X)$ where $X$ is a profinite set and $\cat{A}$ is a functor \[\cat{A}\colon \op{\OX} \to \DMod.\] A {\em morphism of presheaves} $(\Phi, f) \colon (\cat{A},X) \to (\cat{B}, Y)$ comprises a continuous function $f\colon Y \to X$ and a natural transformation 
\[\Phi\colon \cat{A} \circ \cat{pr}_1 \Rightarrow \cat{B}\circ f^\ast\]
between the functors 
\begin{align*}
\cat{A} \circ \cat{pr}_1\colon & \op{\OX}\times \op{\OX[Y]} \to \DMod, & (U,V) &\mapsto \cat{A}(U), \\
\cat{B} \circ f^\ast \colon & \op{\OX}\times \op{\OX[Y]} \to \DMod, & (U,V) &\mapsto \cat{B}(f^{-1}(U)\cap V).
\end{align*}
We will typically denote the map $\cat{A}(U \subseteq V)$ by $\res_U^V$ or as $a \mapsto a|_U$.

A {\em precosheaf of profinite $R$-modules} is a pair $(\cat{M}, X)$ where $X$ is a profinite set and $\cat{M}$ is a functor $\OX \to \PMod$. A {\em morphism of precosheaves} $(\Psi, f) \colon (\cat{N},Y) \to (\cat{M}, X)$ comprises a continuous function $f\colon Y \to X$ and a natural transformation 
\[\Psi\colon  \cat{N}\circ f^\ast\Rightarrow \cat{M} \circ \cat{pr}_1\]
between the functors 
\begin{align*}
\cat{N} \circ f^\ast \colon & {\OX}\times {\OX[Y]} \to \PMod, & (U,V) &\mapsto \cat{N}(f^{-1}(U)\cap V),\\
\cat{M} \circ \cat{pr}_1\colon & {\OX}\times {\OX[Y]} \to \PMod, & (U,V) &\mapsto \cat{M}(U).
\end{align*}
We will typically denote the map $\cat{A}(U \subseteq V)$ by $\cor_U^V$.
\end{defn}
\begin{rmk}
One may readily show that for a fixed space $X$, the category of presheaves $(\cat{A},X)$ over $X$ with morphisms $(\Phi, \id_X)$ covering the identity map is isomorphic to the `usual' presheaf category $\cat{Fun}(\op{\OX}, \DMod)$ of functors and natural transformations over the fixed space $X$.
\end{rmk}
\begin{defn}\label{defn:sheafcosheaf}
A {\em sheaf} of discrete $R$-modules is a presheaf $(\cat{A}, X)$ such that for any compact open sets $U_1,\ldots, U_n \in \OX$ the sequence
\begin{equation}\label{eqn:sheafcondition}
\begin{tikzcd}
0\ar{r}& \cat{A}(U_1\cup\cdots \cup U_n)  \ar{r} & \displaystyle\prod_{i=1}^n \cat{A}(U_i) \ar{r}{p  -  q } & \displaystyle\prod_{1\leq i,j\leq n}  \cat{A}(U_i \cap U_j)
\end{tikzcd}
\end{equation}
is exact, where $p$ denotes the map $\prod_i \res_{U_i\cap U_j}^{U_i}$ and $q$ is the map $\prod_i \res_{U_i\cap U_j}^{U_j}$. The category of sheaves of discrete modules, with morphisms of (pre)sheaves, will be called $\Shf\text{-}R$. 

A {\em cosheaf} of profinite $R$-modules is a precosheaf $(\cat{M}, X)$ such that for any compact open sets $U_1,\ldots, U_n \in \OX$ the sequence
\[\begin{tikzcd}
 \displaystyle\bigoplus_{1\leq i,j\leq n}  \cat{M}(U_i \cap U_j)  \ar{r}{p  -  q }  & \displaystyle\bigoplus_{i=1}^n \cat{M}(U_i) \ar{r} &\cat{M}(U_1\cup\cdots \cup U_n)  \ar{r} & 0 
\end{tikzcd}\]
is exact, where $p$ denotes the map $\sum_{i,j} \cor_{U_i\cap U_j}^{U_i}$ and $q$ is the map $\sum_{i,j} \cor_{U_i\cap U_j}^{U_j}$ . The category of cosheaves of profinite modules, with morphisms of (pre)cosheaves, will be called $R\text{-}\CoShf$. 
\end{defn}

For the usual sheaf theory, one must consider arbitrary families of open sets; here the compactness means the restriction to finite families is sensible. Indeed, since complements of clopen sets remain clopen, one could replace these (co)sheaf equations with rather simpler formulations.
\begin{prop}\label{prop:simplesheaves}
A presheaf $(\cat{A}, X)$ is a sheaf if and only if, for every disjoint union $V=V_1 \sqcup \cdots \sqcup V_n$ of clopen sets, the natural map 
\[\cat{A}( V) \longrightarrow \prod_{i=1}^n \cat{A}(V_i)\]
is an isomorphism.
Dually, a precosheaf is a cosheaf if and only 
\[\bigoplus_{i=1}^n\cat{M}(V_i) \longrightarrow \cat{M}(V)\]
is an isomorphism for every disjoint union $V=V_1 \sqcup \cdots \sqcup V_n$ of clopen sets. 
\end{prop}
\begin{proof}
Given the obvious duality it is enough to prove the result for presheaves. Let $(\cat{A}, X)$ be a presheaf satisfying the given condition and let $U_1,\ldots, U_n \in \OX$. Define $V_i = U_i \smallsetminus (U_1\cup \cdots \cup U_{i-1})$ for each $1\leq i\leq n$. Consider the natural commuting diagram
\[\begin{tikzcd}
0\ar{r}& \cat{A}(U)  \ar{r}  \ar[equals]{d}& \displaystyle\prod_{i=1}^n \cat{A}(U_i) \ar{r}{p  -  q }  \ar{d}{r} & \displaystyle\prod_{1\leq i,j\leq n} \cat{A}(U_i \cap U_j) \ar{d}\\
0\ar{r}& \cat{A}(U)  \ar{r}{\iso} & \displaystyle\prod_{i=1}^n \cat{A}(V_i) \ar{r} & 0
\end{tikzcd}\] 
in \DMod. One immediately sees that the upper left map is injective. Suppose $x=(x_1,\ldots, x_n)$ lies in the kernel of $p-q$. By a diagram chase there exists $x'=(x_1',\ldots, x_n')$ in the image of $\cat{A}(U)$ such that $x-x'$ vanishes when mapped to $\prod_i \cat{A}(V_i)$. 

By induction one finds $x_i = x_i'$ for all $1\leq i\leq n$. Initially, $x_1 = x'_1$ is obvious; for $i>1$ the map
\[\begin{tikzcd}
\cat{A}(U_i) \ar{r}{\iso} & \cat{A}(V_i) \times \cat{A}(U_i \cap (U_1\cup\cdots \cup U_{i-1})) \ar{r} & \cat{A}(V_i) \times \displaystyle\prod_{j=1}^{i-1} \cat{A}(U_i \cap U_j),
\end{tikzcd}\]
which is now known to be injective, combined with the equations
\[\res^{U_i}_{U_i\cap U_j}(x_i - x_i') = p(x-x')_{i,j} = q(x-x')_{i,j} = \res^{U_j}_{U_i\cap U_j}(x_j - x_j') =0\]
for all $j<i$ shows that $x_i = x_i'$. 
\end{proof}

Let $\T$ denote the circle group $\R/\Z$. As is well-known, the Pontryagin duality operation $(-)^\vee = \Hom(-, \T)$ induces exact functors
\[(-)^\vee \colon \PMod \to \op{\DMod}, \quad (-)^\vee \colon \DMod \to \op{\PMod},\]
giving a duality of categories. In view of the obvious duality of our definitions of sheaves and cosheaves, the following proposition becomes almost a tautology.
\begin{prop}
For a sheaf $(\cat{A}, X)$ define the cosheaf $(\cat{A}^\vee, X)$ by $\cat{A}^\vee(U) = (\cat{A}(U))^\vee$.  For a cosheaf $(\cat{M}, X)$ define the sheaf $(\cat{M}^\vee, X)$ by $\cat{M}^\vee(U) = (\cat{M}(U))^\vee$. The assignments $(-)^\vee$ are functors inducing an equaivalence of categories
\[\begin{tikzcd}
\op{(\Shf\text{-}R)} \ar[shift left]{r}{\vee} & \ar[shift left]{l}{\vee} R\text{-}\CoShf.
\end{tikzcd}\]
\end{prop}

To make these rather abstract functors appear more concrete, and relate them to the existing material on families of $R$-modules, we will provide equivalences of these categories with certain categories of topological spaces. We begin with the well-known reformulation of sheaves as \'etale spaces, though we will need to take note that our constructions respect the topology of $R$.

\section{\'Etale spaces of $R$-modules}

\begin{defn}
An {\em \'etale space} over $X$ is a topological space $\cA$ equipped with a local homeomorphism $p\colon \cA \to X$: that is, each point $a\in \cA$ has an open neighbourhood $W$ such that $p|_W\colon W \to p(W)$ is a homeomorphism to an open set $p(W)$ of $X$. A {\em local section} $s\colon U \to \cA$ is a continuous map with $ps = \id_U$, where $U\subseteq X$ is open. A {\em global section} is a local section defined on all of $X$. Write $A_x = p^{-1}(x)$ for the {\em fibre} of $\cA$ over $x\in X$. We refer to the triple $(\cA, p, X)$ as an {\em \'etale bundle}, though we will often abbreviate this simply to ``the \'etale space $\cA$''.

An {\em \'etale space of (right) $R$-modules} comprises:
\begin{itemize}
\item an \'etale bundle $(\cA, p, X)$ with $X$ a profinite set;
\item a global section $o\colon X\to \cA$;
\item a continuous map $\alpha \colon \cA \times_X \cA \to \cA$ commuting with $p$; and
\item a continuous map $\rho \colon \cA \times R\to \cA$ commuting with $p$;
\end{itemize}
such that for each $x\in X$ the operations $\alpha \colon A_x \times A_x \to A_x$ and $\rho \colon A_x \times R \to A_x$ give $A_x$ the structure of a discrete $R$-module with zero element $o(x)$.
\end{defn}
\begin{defn}
A {\em morphism of \'etale spaces of $R$-modules} $(\phi, f) \colon (\cA,p, X) \to (\mathcal{B}, q, Y)$ comprises a continuous function $f\colon Y\to X$ and a continuous function $\phi \colon \cA \times_X Y \to \mathcal{B}$ which restricts to a morphism of $R$-modules $A_{f(y)} \times \{y\} \to B_y$ for each $y\in Y$.

The category of \'etale spaces of $R$-modules and morphisms between them will be denoted $\etale\text{-}R$.
\end{defn}
We record a highly useful lemma about local sections of \'etale spaces. 
\begin{lem}[{\cite[Proposition II.6.1]{maclanemoerdijk}}]\label{lem:etalespaces}
Let $(\cA, p, X)$ be an \'etale bundle. The map $p$ and all local sections $s\colon U \to \cA$ are open maps. Through every point $a\in \cA$ there is at least one local section, and the images of such sections form a base for the topology of $\cA$.

For any local sections $s$ and $t$ of $\cA$, the set $\{x\in X\mid s(x) = t(x)\}$ where both sections are defined and agree is open in $X$. 
\end{lem}
\begin{defn}
Let $\cA$ be an \'etale space of $R$-modules. We denote the {\em space of global sections} $s\colon X\to \cA$ by the symbol
\[\cprod \cA\quad \text{or} \quad \cprod_{x\in X} A_x.\]
We may also refer to it as the `continuous product' of the $A_x$. We endow $\cprod \cA$ with the discrete topology.
\end{defn}
\begin{prop}\label{prop:etalespaces2}
Let $\cA$ be an \'etale space of $R$-modules. For any distinct points $x_1,\ldots, x_n$ of $X$ and any $a_i \in A_{x_i}$ there is a global section $s\colon X\to \cA$ such that $s(a_i) = x_i$. Thus the natural map 
\[\cprod_{x\in X} A_x \to \prod_{x\in X}A_x\]
has dense image, where $\prod_{x\in X} A_x$ has the product topology.
\end{prop}
In particular, for profinite spaces the existence of a local section through each point in Lemma~\ref{lem:etalespaces} may be improved to `global section'.
\begin{proof}
Since $X$ is a profinite space we may find disjoint clopen neighbourhoods $U_i$ of the $x_i$. Perhaps passing to a clopen subset of each $U_i$, there is a local section $s|_{U_i}\colon U_i \to \cA$ through each $a_i$. By defining $s$ to be the restriction of the zero section $o$ to the clopen set $X \smallsetminus (U_1 \cup\cdots \cup U_n)$ we find the required global section. 
\end{proof}
\begin{prop}
Let $\cA$ be an \'etale space of $R$-modules. When $\cprod \cA$ is endowed with the discrete topology, the operations
\begin{align*}
+ \colon \cprod \cA \times \cprod \cA &\to \cprod \cA, & (s,t) & \mapsto (x\mapsto \alpha(s(x), t(x))), \\
\cdot \colon \cprod \cA \times R &\to \cprod \cA, & (s,r) & \mapsto (x\mapsto \rho(s(x), r)),
\end{align*}
give $\cprod \cA$ the structure of a discrete right $R$-module with zero element $o\colon X \to \cA$. 
\end{prop}
\begin{proof}
The well-definedness of this abstract $R$-module structure is immediate. It remains to show that the action of $R$ is continuous. It is sufficient to prove that for any $s\in \cprod \cA$ the ideal 
\[ I = \{r\in R \mid sr = 0\}\]
is open in $R$. 

Now, the image of the zero section $o$ is open in $\cA$, so $\rho^{-1}(o(X))$ is open in $\cA \times R$, and $V=(s \times \id)^{-1}(\rho^{-1}(o(X)))$  is an open set in the profinite space $X\times R$ containing $X \times \{0\}$. By a standard compactness argument there is an open ideal $J \nsgp R$ such that $X \times J \subseteq V$, so the ideal $I$, which contains $J$, is open as required. 
\end{proof}
Now that we know all our objects remain continuous $R$-modules, we are in a position to prove the usual equivalence of \'etale spaces and sheaves. The additions we must make to the standard material are to account for the topologies of the various spaces. 
\begin{theorem}
Let $R$ be a profinite ring with unity. There are functors 
\begin{align*}
\etale\colon &\Shf\text{-}R \to \etale\text{-}R,\\
\Shf\colon &\etale\text{-}R \to \Shf\text{-}R
\end{align*} 
defining an equivalence of categories
\[\begin{tikzcd}
\etale\text{-}R \ar[shift left]{r} &  \ar[shift left]{l} \Shf\text{-}R.
\end{tikzcd}\]
\end{theorem}
\begin{proof}
Let $(\cat{A}, X)$ be a sheaf of discrete $R$-modules. We define the {\em \'etale space of $\cat{A}$} by setting
\[A_x = \varinjlim_{U \ni x} \cat{A}(U), \quad \etale(\cat{A}) = \bigsqcup_{x\in X} A_x.\]
Denote the coprojections in the direct limits by $\res_x^U\colon \cat{A}(U) \to A_x$, or simply $\res_x$ when the open set is clear from context. We endow $\etale(\cat{A})$ with the topology defined by the sub-basic sets
\[ O_s = \{ ( \res_x^U(s), x) \mid x\in U, s\in \cat{A}(U)\}.\]
It is immediate that $(\etale(\cat{A}), p, X)$ is an \'etale bundle, where $p$ is the obvious projection map.

Each fibre $A_x$ inherits an $R$-module structure from the $\cat{A}(U)$. Let us prove continuity of the $R$-action map 
\[\rho\colon \etale(\cat{A}) \times R \to \etale(\cat{A});\]
continuity of the addition map is similar. Suppose $\rho((\res^V_x(t),x),r) \in O_s$. Then $\res_x^V(tr) = \res_x^U(s)$, so there is some $W\subseteq U\cap V$ such that $(tr)|_W=s|_W$. Since the action of $R$ on $\cat{A}(W)$ is continuous, there is an open ideal $I$ of $R$ such that $(tr')|_W = s|_W$ for all $r'\in r+I$. We have
\[((\res_x^V(t),x),r) \in O_{t|_W} \times (r+I) \subseteq \rho^{-1}(O_s),\]
verifying continuity of the action.

For a morphism of sheaves $(\Phi, f)\colon (\cat{A}, X) \to (\cat{B}, Y)$ we define the morphism
\[(\phi, f) = \etale(\Phi, f) \colon (\etale(\cat{A}), X) \to (\etale(\cat{B}), Y) \]
by declaring 
\[\phi(a, y) = (\res_y(\Phi_{U,V}(s)), y)\]
when $a = \res_x^U(s)$.

This map is continuous: for a basic open set $O_t$ of $\etale(\cat{B})$, defined by $t \in \cat{B}(V)$, take any $(a,y)\in \phi^{-1}(O_t)$. Take some $U \in \OX$ and $s \in \cat{A}(U)$ with $a = \res_x^U(s)$; then the local sections $\Phi_{U,V}(s)$ and $t|_{f^{-1}(U)\cap V}$ of $\etale(\cat{B})$ agree at $y$, so by Lemma~\ref{lem:etalespaces} there is some $W\in \OX[Y]$ such that $\Phi_{U,V}(s)|_W = t|_W$. Then 
\[(O_s \times W)\cap (\etale(\cat{A}) \times_X Y) \]
is an open neighbourhood of $(a,y)$ in $\phi^{-1}(O_t)$.

Let us define the functor in the other direction. For an \'etale space of $R$-modules $(\cA, p, X)$ we define the sheaf $\Shf(\cA)$ by 
\[\Shf(\cA)(U) = \cprod_{u\in U} A_u = \cprod p^{-1}(U).\]
Given a morphism of \'etale spaces $(\phi,f)\colon (\cA, p, X) \to (\mathcal{B}, q, Y)$ we associate the morphism of sheaves $(\Phi, f) = \Shf(\phi, f)$ defined by 
\[\Phi_{U,V}(s)(y) = \phi(s(f(y)),y)\]
for $y\in f^{-1}(U)\cap V$ and $s\in \cprod_{u\in U} A_u$. 

To establish the equivalence of categories we define natural isomorphisms $\etale \circ \Shf \Rightarrow \mathsf{id}$ and $\mathsf{id} \Rightarrow \Shf\circ\etale$. For an \'etale space $\cA$ the natural isomorphism is defined by 
\begin{align*}
\etale(\Shf(\cA)) & \to  \cA \\
( \res_x(s), x) &\mapsto s(x)
\end{align*}
for $s \in \cat{A}(U)$. It follows from Lemma~\ref{lem:etalespaces} that this is an isomorphism.

In the other direction we define a natural isomorphism of functors 
\[ \cat{J} \colon \cat{A} \Rightarrow \Shf(\etale(\cat{A}))\]
on \OX, which readily gives an isomorphism of sheaves 
\[  \Phi_{U,V} = \res_{U\cap V}^V \circ \cat{J}_U \colon \cat{A}(U) \to \Shf(\etale(\cat{A}))(U\cap V).\]
The natural transformation $\cat{J}$ is defined by 
\begin{align*}
\cat{J}_U \colon \cat{A}(U) &\to \cprod_{u\in U} \etale(\cat{A})_u,\\
a & \mapsto (u \mapsto (\res_u(a), u))
\end{align*}
for $a\in \cat{A}(U)$. It remains to show that $\cat{J}_U$ is an isomorphism for each $U$. 

For injectivity, suppose $\res_u(a) = \res_u(b)$ for all $u\in U$. Then for each $u$ there is some $V_u \in \OX[U]$ such that $a|_{V_u} = b|_{V_u}$. Since $U$ is compact, it is covered by finitely many of the $V_u$, whence the sheaf condition \eqref{eqn:sheafcondition} shows $a=b$.

For surjectivity, take some section $s\colon U \to \etale(\cat{A})$. For each $u\in U$ there is some $a_u \in \cat{A}(U)$ such that $s(u) = \res_x(a_u)$ by Proposition~\ref{prop:etalespaces2}. By Lemma~\ref{lem:etalespaces} there is some $V_u \in \OX[U]$ such that $s|_{V_u} = a_u|_{V_u}$. By compactness of $U$ we may find a partition of $U$ into clopen sets $W_1, \ldots, W_n$ such that each $W_i$ is contained in some $V_{u_i}$. By \eqref{eqn:sheafcondition} there is now some $a\in \cat{A}(U)$ such that $a|_{W_i} = s|_{W_i}$ for all $i$, whence $s = \cat{J}_U(a)$ as required. 
\end{proof}
One should note that the functor $\cprod\colon \etale\text{-}R \to \DMod$ is recovered as the composition
\[\begin{tikzcd}[column sep = large]
\etale\text{-}R \ar{r}{\Shf} & \Shf\text{-}R \ar{r}{\eval_X} & \DMod
\end{tikzcd}\]
where $\eval_X$ is the functor $\cat{A} \mapsto \cat{A}(X)$.

\section{`Sheaves' of profinite $R$-modules}
We must now confront an unfortunate fact of historical notation: the objects which it is our business to exhibit as dual to \'etale spaces of $R$-modules exist in the literature under the name ``sheaves of profinite $R$-modules''. These things are definitely not sheaves in the geometric sense; they are not functors on categories of open sets, but families of profinite $R$-modules which are in some sense continuously indexed by a profinite space $X$. Instead of conjuring some new name for them (`co\'etale space' springs to mind) we will continue to call them `sheaves', with the inverted commas serving as a reminder of the type of object we are considering. We recall the definition. 
\begin{defn}\label{defn:sheafpromodules}
Let $R$ be a profinite ring. A {\em `sheaf' of $R$-modules} consists of a triple $(\cM, p, X)$ with the following properties.
\begin{itemize}
\item $\cM$ and $X$ are profinite spaces and $p\colon\cM\to X$ is a continuous surjection.
\item Each {\em fibre} $M_x=p^{-1}(x)$ is endowed with the structure of a profinite $R$-module such that the maps
\[R\times\cM \to \cM, \quad (r,m)\to r\cdot m \]
\[\cM \times_X \cM \to \cM, \quad (m,n)\to m+n \]
are continuous.
\end{itemize} 
A {\em morphism of `sheaves'} $(\psi,f)\colon (\mathcal{N}, q, Y)  \to (\cM, p, X)$ consists of continuous maps $\psi: {\mathcal{N}}\to\cM$ and $f\colon Y\to X$ such that $p\psi = fq$ and such that the restriction of $\psi$ to each fibre is a morphism of $R$-modules $N_y\to M_{f(y)}$.
\end{defn}
We often contract ``the `sheaf' $(\cM, p, X)$'' to simply ``the `sheaf' $\cM$''. Regarding an $R$-module as a `sheaf' over the one-point space one may talk of a `sheaf' morphism from a `sheaf' to an $R$-module. The category of `sheaves' and morphisms will be called $\PShf$.
\begin{defn}
A {\em profinite direct sum} of a `sheaf' $\cM$ consists of an $R$-module $\bigboxplus_X\cM$ and a `sheaf' morphism $\omega\colon \cM\to \bigboxplus_X \cM$ (sometimes called the `canonical morphism') such that for any profinite $R$-module $P$ and any `sheaf' morphism $\beta\colon\cM\to P$ there is a unique morphism of $R$-modules $\tilde\beta\colon\bigboxplus_X\cM\to P$ such that $\tilde\beta\omega=\beta$.
\end{defn}
The basic properties of `sheaves' of profinite $R$-modules are known. They are essentially given in \cite[Chapter~5 and Section~9.1]{ribes} in the language of  free profinite products of profinite groups, rather than for $R$-modules. An exposition purely for modules appeared in \cite[Appendix~A]{wilkes}.

Our task now is to prove that `sheaves' of profinite $R$-modules are in fact equivalent to cosheaves of profinite modules in the sense of Definition~\ref{defn:sheafcosheaf}. 

For each `sheaf' $(\cM, p, X)$ we may define a cosheaf $\CoShf(\cM)$ by 
\[\CoShf(\cM)(U) = \bigboxplus_{u\in U} M_u.\]
Given a `sheaf' morphism $(\psi,f)\colon (N,q,Y) \to (M,p,X)$ we have a morphism of cosheaves 
\[\Psi_{U,V} \colon \bigboxplus_{y \in f^{-1}(U)\cap V} N_y \to \bigboxplus_{u\in U} M_u  \]
induced by the morphism of `sheaves' $q^{-1}( f^{-1}(U)\cap V) \to p^{-1}(U) \to \bigboxplus_{u\in U}M_u$.

Conversely, given a cosheaf $\cat{M}$, we seek to define a `sheaf' of profinite $R$-modules
\[\cM = \coet(\cat{M}) = \bigsqcup_{x\in X} M_x\]
where $M_x$ is the `stalk'
\[M_x = \varprojlim_{U \ni x} \cat{M}(U).\]
Let $\cor_x^U \colon M_x \to \cat{M}(U)$ be the projection map from the inverse limit. Impose on $\cM$ the topology generated by the sub-basic sets
\[O_{U,W} = \{(m,x) \mid x\in U, \cor_x^U(m) \in W\}\]
where $U\in \OX$ and $W\subseteq \cat{M}(U)$ is open. To show that $\cM$ is a profinite set under this topology, observe that this topology is the same as the topology induced on \cM\ by the inclusion 
\[i\colon \cM \to \prod_{U \in \OX} \cat{M}(U) \times X, \quad i(m,x) = ((\cor_x^U(m)), x),\]
where we set $\cor_x^U = 0$ if $x\notin U$. The latter space being profinite, it suffices to show that $i(\cM)$ is closed. Continuity of the $R$-module operations on $\cM$ is inherited from the space on the right hand side above. 

Take a point $((m_U), x)$ outside $i(\cM)$. Then either there is some $V\notni x$ with $m_V \neq 0$ or there is no $m_x\in M_x$ with $m_U = \cor_x^U(m_x)$ for all $U\ni x$. In the former case, we have
 \[ ((m_U), x) \in \mathrm{pr}_V^{-1}(\cat{M}(V)\smallsetminus\{0\}) \times (X\smallsetminus V) \]
as an open set separating $((m_U),x)$ from $i(\cM)$. In the latter case, there are clopen sets $U\supseteq V \ni x$ such that $\cor_V^U(m_V) \neq m_U$. Take disjoint clopen subsets $W$ and $W'$ of $\cat{M}(U)$ such that $m_V \in W$ and $m_U \in W'$. Then 
\[((m_U),x) \in   ( \mathrm{pr}_V^{-1}((\cor_V^U)^{-1}(W)) \cap \mathrm{pr}_U^{-1}(W') ) \times V\]
is open and contains no point of $i(\cM)$. Thus \cM\ is indeed a profinite space.  

Given a morphism of cosheaves $(\Psi, f)\colon (\cat{N}, Y) \to (\cat{M}, X)$ we may define a `sheaf' morphism $(\psi, f)$ by noting that the natural maps
\[\begin{tikzcd} N_y \ar{r}{\cor_y^{f^{-1}(U)}} & \cat{N}(f^{-1}(U)) \ar{r}{\Psi_{U,Y}} & \cat{M}(U) \end{tikzcd}\]
define a cone on the $\cat{M}(U)$ and therefore give a natural morphism of profinite $R$-modules $N_y \to M_{f(y)}$ for each $y\in Y$. This map is a continuous function
\[\psi \colon \coet(\cat{N}) \to \coet(\cat{M});\]
for if $O_{U,W}$ is a sub-basic open set in $\coet(\cat{M})$ then 
\[\psi^{-1}(O_{U,W}) = O_{f^{-1}(U), \Psi_{U,Y}^{-1}(W)}\]
is open in $\coet(\cat{N})$.

\begin{theorem}
Let $R$ be a profinite ring with unity. There are functors 
\begin{align*}
\coet\colon &R\text{-}\CoShf \to \PShf \\
\CoShf\colon &\PShf \to R\text{-}\CoShf
\end{align*} 
defining an equivalence of categories
\[\begin{tikzcd}
R\text{-}\CoShf \ar[shift left]{r} &  \ar[shift left]{l}\PShf.
\end{tikzcd}\]
\end{theorem}
\begin{proof}{}
We establish the equivalence of categories by giving natural isomorphisms 
\[\cat{id}\Rightarrow \coet \circ \CoShf, \quad \cat{J}\colon \CoShf\circ \coet \Rightarrow \cat{id}.\]
Let $\cat{M}$ be a cosheaf on $X$ and write $(\cM, p, X) = \coet(\cat{M})$. For each $U\in\OX$ the map
\[p^{-1}(U) \to \cat{M}(U), \quad (m,x) \mapsto \cor_x^U(m)\]
gives, by the universal property of profinite direct sums, a natural map
\[\cat{J}_U \colon \CoShf(\cM)(U) =\bigboxplus p^{-1}(U) \to \cat{M}(U).\]
Observe that we may write $\cat{M}(U)$ as the inverse limit of the modules $\bigoplus_i \cat{M}(V_i)$ indexed over the clopen partitions $U= \bigsqcup_i V_i$ of $U$. All these modules are naturally isomorphic to $\cat{M}(U)$ by the cosheaf condition in Proposition~\ref{prop:simplesheaves}. It now follows from \cite[Proposition~A.11]{wilkes} that $\cat{M}(U)$ is naturally isomprhic to the profinite direct sum of the limits $M_x = \varprojlim_{V\ni x}\cat{M}(V)$---that is, $\cat{J}_U$ is an isomorphism for all $U$. 

In the other direction, take a `sheaf' \cM. The natural maps $M_x \to \bigboxplus_{u\in U} M_u$ being injective for all $U\ni x$ \cite[Proposition~A.6]{wilkes}, an application of \cite[Proposition~A.11]{wilkes} shows that there are natural isomorphisms \[M_x \iso \varprojlim_{U\ni x} \bigboxplus_{u\in U} M_u\] for each $x\in X$. This yields a natural bijection
\[\cM \iso \coet(\cat{M})\]
which is readily seen to  be continuous, the preimage of a basic open set $O_{U,W}$ being the preimage of $W$ under the canonical continuous map $p^{-1}(U) \to \bigboxplus p^{-1}(U)$. 
\end{proof}

\section{Pontryagin duality for sheaves}

The union of the equivalences of categories proved in the previous three sections is the following diagram. 
\begin{theorem}\label{thm:maintheorem}
Let $R$ be a profinite ring with unity. There is a diagram of functors
\[\begin{tikzcd}[column sep = large]
\op{(\etale\text{-}R)} \ar[leftrightarrow]{rd}[swap]{\iso} \ar{rrd}{\cprod} & & \\
&  \op{(\Shf\text{-}R)}  \ar[leftrightarrow]{d}{\iso}[swap]{\vee} \ar{r}[swap]{\eval_X} & \op{(\DMod)} \ar[leftrightarrow]{d}{\iso}[swap]{\vee}\\
& R\text{-}\CoShf  \ar{r}{\eval_X} & \PMod \\
\PShf  \ar[leftrightarrow]{ur}{\iso} \ar{rru}[swap]{\bigboxplus}& &
\end{tikzcd}\]
which commutes up to natural isomorphism.
\end{theorem}

By filling in the `open' side of the diagram above we may define a duality between \'etale spaces and `sheaves' of profinite modules. The expected relationship on fibres holds since Pontryagin duality commutes with direct and inverse limits. 
\begin{defn}
Given an \'etale space of $R$-modules $(\cA, p, X)$ define the {\em dual `sheaf'} 
\[\widecheck{\cA} = \coet(\Shf(\cA)^\vee).\]
For a `sheaf' of profinite $R$-modules $(\cM, p, X)$ define the {\em dual \'etale space}
\[\widecheck{\cM}=\etale(\CoShf(\cM)^\vee).\]
\end{defn}
\begin{theorem}\label{thm:maintheorem2}
The functors $(\widecheck{-})$ define an equivalence of categories 
\[\begin{tikzcd} \op{(\etale\text{-}R)} \ar[shift left]{r} & \ar[shift left]{l} \PShf.\end{tikzcd} \]
There are natural isomorphisms of fibres 
\[ \widecheck{\cA}_x \iso A_x^\vee, \quad \widecheck{\cM}_x \iso M_x^\vee\]
and natural isomorphisms
\[\bigboxplus \widecheck{\cA} \iso (\cprod \cA)^\vee,\quad  \cprod \widecheck{\cM} \iso (\bigboxplus \cM)^\vee\]
of profinite direct sums and continuous products. 
\end{theorem}
We immediately inherit two important corollaries about direct limits of \'etale spaces from the category of `sheaves' of $R$-modules. 
\begin{clly}[{\cite[Proposition A.12]{wilkes}}]
Every \'etale space of $R$-modules may be written as a direct limit of finite \'etale spaces of $R$-modules. 
\end{clly}
\begin{clly}[{\cite[Proposition A.8]{wilkes}}]
For a directed system of \'etale spaces $(\cA_i, p_i, X_i)$, the direct limit $(\varinjlim \cA_i, p, \varprojlim X_i)$ is an \'etale space of $R$-modules and 
\[\varinjlim \cprod \cA_i \iso \cprod \varinjlim \cA_i.\]
\end{clly}
The equivalence of \'etale spaces and {\it bona fide} sheaves also makes it easy to combine them with functors on the category of discrete $R$-modules. 
\begin{theorem}\label{thm:functors}
Let $R$ and $S$ be profinite rings with unity and let $\cat{F}\colon \DMod \to \DMod[S]$ be an additive functor which commutes with direct limits. There is an induced functor $\cat{F}\colon \etale\text{-}R \to  \etale\text{-}S$ such that for every \'etale space $\cA$ of $R$-modules, the fibre $\cat{F}(\cA)_x$ is naturally isomorphic to $\cat{F}(A_x)$ and there are natural isomorphisms
\[\cprod \cat{F}(\cA) \iso \cat{F}(\cprod \cA).\]
\end{theorem}
\begin{proof}
Since \cat{F}\ is additive, and therefore commutes with finite products, by Proposition~\ref{prop:simplesheaves} we find that the presheaf
\[\cat{FA}(U) = \cat{F}(\cat{A}(U)) \]
defines a sheaf for any sheaf $\cat{A}$. For an \'etale space \cA\ of $R$-modules we may thus define 
\[\cat{F}(\cA) = \etale(\cat{F}\Shf(\cA)).\]
The natural isomorphism 
\[\cprod \cat{F}(\cA) \iso \cat{F}(\cprod \cA)\]
is immediate. Because $\cat{F}$ commutes with direct limits we also find
\[\cat{F}(\cA)_x = \varinjlim_{U\ni x} \cat{F}\Shf(\cA)(U) = \cat{F}(\varinjlim \Shf(\cA)(U)) \iso \cat{F}(A_x)\]
as required.
\end{proof}
A particularly important family of functors satisfying the conditions of the above theorem are Ext-functors. Note that the image of such an Ext-functor is a discrete torsion abelian group, which may be regarded as an object of \DMod[\widehat{\Z}]. 
\begin{clly}
Let $Z$ be a profinite (right) $R$-module and let $\cA$ be an \'etale space of $R$-modules. Then there are natural isomorphisms of cohomological $\delta$-functors 
\[\Ext^\bullet_R(Z, \cprod \cA) \iso \cprod_{x\in X} \Ext^\bullet_R(Z,A_x).\]
\end{clly}
Having established duality of `sheaves' of profinite modules and \'etale spaces of $R$-modules, many useful results concerning the former may be reformulated in the dual category. To conclude the paper, let us give one key example which is likely to prove useful. From now on we let $R= \Zpiof{G}$ be the completed group ring of a profinite group $G$ for some set of primes $\pi$. Recalling that the categories of right and left $R$-modules are canonically isomorphic for such a ring $R$, we will drop the distinction and regard all modules as left modules. An equivalent name for an $R$-module for this ring $R$ would be ``discrete $\pi$-primary $G$-module''.

Let the profinite group $G$ act on the profinite set $Y$ and let $q\colon Y \to G\lqt Y$ be the quotient map. There is an isomorphism \cite[Proposition~A.14]{wilkes} 
\[\Zpiof{Y}\iso \bigboxplus_{\underline y \in G \lqt Y} \Zpiof{q^{-1}(\underline y)}.\]
Since the functor $\Ext_R^\bullet((-)^\vee, A)$ satisfies the conditions of Theorem~\ref{thm:functors} we have a natural isomorphism 
\[\Ext_R^\bullet(\Zpiof{Y}, A) \iso \cprod_{\underline y \in G\lqt Y} \Ext_R^\bullet(\Zpiof{q^{-1}(\underline y)}, A)\]
for each discrete torsion $ \Zpiof{G}$-module $A$.
When $q$ admits a continuous section $\sigma \colon G \lqt Y \to Y$, so that the family of stabilizers 
\[G_{\underline y} = \stab(\sigma(\underline y))\]
is continuously indexed by $G\lqt Y$ (see \cite[Section 5.2]{ribes}), this isomorphism takes the form 
\[\Ext_R^\bullet(\Zpiof{Y}, A) \iso \cprod_{\underline y \in G\lqt Y} \Ext_R^\bullet(\Zpiof{G/G_{\underline y}}, A)\]
and therefore, by Shapiro's lemma (see \cite[Proposition~1.23]{wilkes}) we have a natural isomorphism 
\[\Ext_R^\bullet(\Zpiof{Y}, A) \iso \cprod_{\underline y \in G\lqt Y} H^\bullet(G_{\underline y}, A).\]

Now suppose $G$ acts on a pro-$\pi$ tree $T$, so that $G$ has actions on the profinite vertex and edge spaces $V(T)$ and $E^*(T) = T/V(T)$ respectively. By definition of a pro-$\pi$ tree \cite[Section 2.4]{ribes} there is a short exact sequence of $\Zpiof{G}$-modules
\[\begin{tikzcd}
0 \ar{r}& \Zpiof{E^*(T), \ast} \ar{r}& \Zpiof{V(T)}\ar{r}& \Z[\pi] \ar{r}& 0.
\end{tikzcd}\]
Applying the cohomological Ext-functors yields a long exact sequence
\[\begin{tikzcd}[row sep = tiny]
\to H^n(G, A) \ar{r}& \displaystyle\cprod_{\underline v\in G\lqt V(T)} \Ext_R^n(\Zpiof{G/G_{\underline v}}, A) 
\ar[phantom]{d}[coordinate, name=Z]{""}
\ar[rounded corners,to path={ -- ([xshift=2ex]\tikztostart.east)
|- (Z) [near end]\tikztonodes
-| ([xshift=-2ex]\tikztotarget.west)
-- (\tikztotarget)} ]{d} 
 &\\
& \displaystyle\cprod_{\underline e \in G\lqt E^*(T)} \Ext_R^n(\Zpiof{G/G_{\underline e}}, A) \ar{r} &  H^{n+1}(G,A)  \to 
\end{tikzcd}\]
for every discrete $\pi$-primary $G$-module $A$. When the quotient map admits a section, we obtain the desired cohomological Mayer-Vietoris sequence dual to \cite[Theorem 9.4.1]{ribes}.
\begin{theorem}
Let a profinite group $G$ act on a pro-$\pi$ tree $T$ and suppose the quotient map $T \to G\lqt T$ admits a section $\sigma$. Then there is a natural long exact sequence
\[\begin{tikzcd}[row sep = tiny]
\cdots \to H^n(G, A) \ar{r}& \displaystyle\cprod_{\underline v\in G\lqt V(T)} H^n(G_{\sigma(\underline v)}, A) 
\ar[phantom]{d}[coordinate, name=Z]{""}
\ar[rounded corners,to path={ -- ([xshift=2ex]\tikztostart.east)
|- (Z) [near end]\tikztonodes
-| ([xshift=-2ex]\tikztotarget.west)
-- (\tikztotarget)} ]{d} 
 &\\
& \displaystyle\cprod_{\underline e \in G\lqt E^*(T)} H^n(G_{\sigma(\underline e)}, A) \ar{r} &  H^{n+1}(G,A)  \to \cdots
\end{tikzcd}\]
for any discrete $\pi$-primary $G$-module $A$. 
\end{theorem}
\bibliographystyle{alpha}
\bibliography{Prof_Sheaf.bib}
\end{document}

%% file: macros.tex
\usepackage{amsmath}
\usepackage{amsfonts}
\usepackage{amssymb}
\usepackage{amsthm}
\usepackage{mathtools}
\usepackage{tikz}
\usepackage{tikz-cd}
\usepackage{enumerate}
\usepackage{layout}
\usepackage{mathabx}
\usepackage{bm}
\usetikzlibrary{patterns}

\DeclareMathOperator{\id}{id}
\DeclareMathOperator{\stab}{stab}
\DeclareMathOperator{\eval}{eval}
\DeclareMathOperator{\Hom}{Hom}
\DeclareMathOperator{\Ext}{Ext}
\DeclareMathOperator{\res}{res}
\DeclareMathOperator{\cor}{cor}
\newcommand{\iso}{\ensuremath{\cong}}
\newcommand{\Z}[1][]{\ensuremath{\mathbb{Z}_{#1}}}
\newcommand{\R}{\ensuremath{\mathbb{R}}}
\newcommand{\T}{\ensuremath{\mathbb{T}}}

\newcommand{\cat}[1]{\ensuremath{\mathsf{#1}}}
\newcommand{\Shf}{\cat{Shf}}
\newcommand{\PShf}{\ensuremath{R\text{-\sffamily\upshape P`Shf'}}}
\newcommand{\CoShf}{\cat{CoShf}}
\newcommand{\etale}{\cat{\acute{E}tale}}
\newcommand{\coet}{\cat{Co\acute{E}tale}}
\newcommand{\DMod}[1][R]{\ensuremath{\cat{DMod}\text{-}#1}}
\newcommand{\PMod}{\ensuremath{R\text{-}\cat{PMod}}}
\newcommand{\OX}[1][X]{\ensuremath{\cat{O_c}(#1)}}
\newcommand{\op}[1]{\ensuremath{{#1}^{\mathrm{op}}}}

\newcommand{\cA}{\ensuremath{\mathcal{A}}}
\newcommand{\cM}{\ensuremath{\mathcal{M}}}

\makeatletter
\newcommand{\superimpose}[2]{{%
  \ooalign{%
    \hfil$\m@th#1\@firstoftwo#2$\hfil\cr
    \hfil$\m@th#1\@secondoftwo#2$\hfil\cr
  }%
}}
\makeatother

\newcommand{\cprod}{\mathop{\mathpalette\superimpose{{\cdot}{\prod}}}}

\newcommand{\nsgp}[1][]{\ensuremath{\triangleleft_{#1}}}
\newcommand{\Zpiof}[1]{{\ensuremath{\Z[{\pi}][\![#1]\!]}}}
\newcommand{\lqt}{\backslash}

%% file: environs.tex
\newtheorem{theorem}{Theorem}[section]
\newtheorem{prop}[theorem]{Proposition}

\newtheorem{lem}[theorem]{Lemma}
\newtheorem{clly}[theorem]{Corollary}
\theoremstyle{definition}
\newtheorem{defn}[theorem]{Definition}
\newtheorem*{cnv}{Convention}
\newtheorem{rmk}[theorem]{Remark}